\documentclass[12pt,a4paper,reqno]{amsart}
\usepackage{amsmath,amsfonts,amssymb,amsthm,amscd}
\usepackage[english]{babel}
\usepackage[mathscr]{eucal}
\usepackage{graphicx}

\numberwithin{equation}{section}
\renewcommand{\a}{\alpha}
\renewcommand{\b}{\beta}
\newcommand{\g}{\gamma}
\newcommand{\G}{\Gamma}
\renewcommand{\d}{\delta}
\newcommand{\D}{\Delta}

\renewcommand{\l}{\lambda}

\newcommand{\n}{\nu}
\renewcommand{\o}{\omega}
\renewcommand{\O}{\Omega}
\renewcommand{\r}{\rho}

\newcommand{\e}{\varepsilon}
\newcommand{\f}{\varphi}
\newcommand{\F}{\Phi}

\newcommand{\p}{\psi}
\renewcommand{\P}{\Psi}

\newcommand{\C}{{\mathbb C}}

\newcommand{\N}{{\mathbb N}}
\newcommand{\R}{{\mathbb R}}

\newcommand{\Hc}{{\mathcal H}}

\newcommand{\Vc}{{\mathcal V}}

\newcommand{\curl}{{\rm curl}\,}

\newcommand{\diver}{{\rm div}\,}
\newcommand{\inte}{{\rm int}\,}
\newcommand{\pd}{\partial}

\newcommand{\supp}{\operatorname{supp\,}}

\newcommand{\loc}{\operatorname{{loc}}}

\newtheorem{theorem}{Theorem}[section]
\newtheorem{proposition}[theorem]{Proposition}
\newtheorem{lemma}[theorem]{Lemma}
\newtheorem{corollary}[theorem]{Corollary}

\theoremstyle{definition}
\newtheorem{definition}[theorem]{Definition}
\newtheorem{assumption}[theorem]{Assumption}

\theoremstyle{remark}
\newtheorem{remark}[theorem]{Remark}
\newtheorem{example}[theorem]{Example}

\newcommand{\na}{\nabla}

\begin{document}


\title[Two Dimensional Incompressible Viscous Flow]
      {Two Dimensional Incompressible Viscous Flow Around a Thin Obstacle Tending to a Curve}
\author[C. Lacave]{Christophe Lacave}
\address[C. Lacave]{Universit\'e de Lyon\\
Universit\'e Lyon1\\
INSA de Lyon, F-69621\\
Ecole Centrale de Lyon\\
CNRS, UMR 5208 Institut Camille Jordan\\
Batiment du Doyen Jean Braconnier\\
43, blvd du 11 novembre 1918\\
F - 69622 Villeurbanne Cedex\\
France}

\email{lacave@math.univ-lyon1.fr}

\date{\today}

\begin{abstract}
In \cite{lac_euler} the author considered the two dimensional Euler equations in the exterior of a thin obstacle shrinking to a curve and determined the limit velocity. In the present work, we consider the same problem in the viscous case, proving convergence to a solution of the Navier-Stokes equations in the exterior of a curve. The uniqueness of the limit solution is also shown.
\end{abstract}

\maketitle

\section{Introduction}

The present paper studies  the influence of a thin material obstacle on the behavior of two-dimensional incompressible viscous flow. The study of flow past slender body is a classical problem in fluid mechanics and it presents a rich literature on experiments and simulations, specially around a flat plane (see for example \cite{phy_1,phy_2,phy_3,phy_4,phy_5,phy_6}). The goal of this work is to establish existence and uniqueness outside a curve. The mathematical study of the problem of small obstacles in incompressible flows has been initiated by Iftimie, Lopes Filho and Nussenzveig Lopes \cite{ift_lop,ift_lop_2,ift_lop_3,lop} and continued in \cite{lac_euler}. Let  $\O_\e$ be a small connected and simply connected bounded open set in $\R^2$.  In all these papers, the initial data consists in the initial vorticity $\o_0$  and the circulation $\g$ of the initial velocity around the boundary of the obstacle. Both $\o_0$ (supposed to be smooth and compactly supported) and the circulation $\g$ are assumed to be independent of $\e$. Given the geometry of the obstacle $\O_\e$, the two previous quantities uniquely determine the initial velocity field $u_0^\e$ (divergence-free, tangent to the boundary and vanishing at infinity). With this initial data, the problem we consider here is to determine the limit of the solutions of the Navier-Stokes equations in the exterior of  $\O_\e$ when the obstacle $\O_\e$ shrinks to a curve as $\e\to 0$. In \cite{ift_lop} the authors studied the vanishing obstacle problem for incompressible, ideal, two-dimensional flow when the obstacle homothetically shrinks to a point. It is proved there that the limit velocity satisfies a modified Euler equation containing an additional term which is a fixed Dirac mass of strength $\g$ in the point where the obstacle shrinks to. In \cite{lac_euler}, the author treated the same problem in the case when the obstacle shrinks to a curve $\G$ instead of a point. In this case, the additional term is of the form $g_\o \d_\G$ where $\d_\G$ is the Dirac mass of the curve. The density $g_\o$ is explicitly computed in  \cite{lac_euler} and depends on the vorticity  and the circulation $\g$. It can be seen as the jump across $\G$ of the velocity field that is divergence free, tangent to $\G$, vanishing at infinity and with curl $\o$ in $\R^2\setminus\G$. The case of several obstacles, one of them shrinking to a point, was treated in \cite{lop}. The two dimensional viscous case where the obstacle shrinks homothetically to a point was studied in \cite{ift_lop_2}, where it is proved that in the case of small circulation the limit equations are always the Navier-Stokes equations where the additional Dirac mass appears only on the initial data. This is due to the fact that the circulation of the initial velocity on the boundary of the obstacle does vanish for $t>0$ when we consider the no-slip boundary condition.

Here we assume that the obstacle shrinks to a curve and we pass to the limit in the Navier-Stokes equations in the exterior of this obstacle. We prove that the limit equations are the Navier-Stokes equations in the exterior of the curve and have a unique solution in a suitable sense. As we shall see in Section \ref{evanescent}, the initial data for the limit velocity is not square-integrable since it behaves as $\frac{x^\perp}{2\pi|x|^2}$ at infinity. For such an initial data we define a solution of the Navier-Stokes equations as a vector field verifying the equation in the sense of distributions and such that the difference between the solution and a fixed smooth vector field behaving like $\frac{x^\perp}{2\pi|x|^2}$ at infinity has the regularity expected from a Leray solution (see  Definition \ref{def} for the precise definition).

More precisely, let $\O_\e$ be a simply connected smooth bounded domain such that $\O_\e$ shrinks to a curve $\G$ as $\e\to 0$ in the sense of Section \ref{evanescent}. The aim of this paper is to prove the following theorem.
\begin{theorem}\label{intro} Let $\o_0$ and $\g$ be independent of $\e$ as defined above. Let $u^\e$ be the solution of the Navier-Stokes equations on $\Pi_\e\equiv \R^2\setminus \overline{\O_\e}$ with initial velocity $u_0^\e$ (see (\ref{NS}) below) and denote by $Eu^\e$ the extension of $u^\e$ to $\R^2$ with values $0$ on $\O_\e$. 
Then $\{ E u^\e\}$ converges in $L^2_{\loc}([0,\infty)\times (\R^2\setminus\G))$ to a solution of the Navier-Stokes equations in $\R^2\setminus\G$ (in the sense of Definition \ref{def}).
\end{theorem}

The initial vorticity of this limit solution is $\o_0+g_\o \d_\G$ and the initial velocity is given by the relation 
$$u_0=K[\o_0]+\a H,$$
with $K$ and $H$ defined in \eqref{Kbis} and \eqref{Hbis} depend only on the $\G$ shape, and with $\a=\g + \int \o_0$. Then, this initial velocity is explicitly given in terms of $\o_0$ and $\g$ and can be viewed as the divergence free vector field which is tangent to $\G$, vanishing at infinity, with curl in $\R^2\setminus\G$ equal to $\o$ and with circulation around the curve $\G$ equal to $\g$. This velocity is  blowing up at the endpoints of the curve $\G$ as the inverse of the square root of the distance and has a jump across $\G$. In fact one can also characterize $g_\o$ as the jump of the tangential velocity across $\G$.

Moreover, for such initial data, we also show that a solution of the Navier-Stokes equations in $\R^2\setminus\G$ (in the sense of Definition \ref{def}, which means that the difference between the solution and a fixed smooth vector field behaving like $\frac{x^\perp}{2\pi|x|^2}$ at infinity has the regularity expected from a Leray solution) is unique (see Proposition \ref{prop : unicity} for the precise statement).

The existence of solutions in the Navier-Stokes equations has been studied in general domains in \cite{gallagher} for the dimension two or three for square-integrable data, and in \cite{monniaux} for the dimension three and $H^{\frac12}$ initial data. Kozono and Yamazaki \cite{koz} treated the case of $L^{2,\infty}$ data but for exterior domains which are smooth. A byproduct of Theorem \ref{intro} is the existence and uniqueness of solutions of the Navier-Stokes equations on $\R^2\setminus\G$ in a case which is not covered in previous work. Indeed, the result of \cite{gallagher} does not apply because the initial data of our limit velocity is not square-integrable at infinity. Our extension from square-integrable velocities to velocities that decay like $1/|x|$ is physically meaningful: it allows nonvanishing initial circulation around the obstacle, something which can happen in impulsively started motions. On the other hand, our initial data $u_0$ satisfies the smallness condition of Kozono and Yamazaki \cite{koz} (see (\ref{small_cond}) below), but the domain $\R^2\setminus\G$ is not smooth, as required in \cite{koz}.

The remainder of this work is organized as follows. We introduce in Section 2 a family of conformal mappings between the exterior of $\O_\e$ and the exterior of the unit disk, allowing the use of explicit formulas for basic harmonic fields and the Biot-Savart law. Moreover, we formulate the flow problem in the exterior of a vanishing obstacle and we study the asymptotic behavior of the initial data. In Section 3 we find {\it a priori} estimates which will be used in Section 4 to prove compactness in space-time and perform the passage to the limit stated in Theorem \ref{intro}. In Section 5 the uniqueness of the Navier-Stokes equations on the exterior of a curve is established.

For the sake of clarity, the main notations are listed in an appendix at the end of the paper.

\section{Flow in an exterior domain}

\subsection{Conformal mapping}\

Let $D=B(0,1)$ and $S=\pd D$.
In what follows we identify $\R^2$ with the complex plane $\C$.

We begin this section by recalling some basic definitions on the curve.

\begin{definition} We call a {\it Jordan arc} a curve $C$ given by a parametric representation $C:\f(s)$, $0\leq s\leq 1$ with $\f$ an injective ($=$one-to-one) function, continuous on $[0,1]$. An {\it open Jordan arc} has a parametrization $C:\f(s)$, $0< s<1$ with $\f$ continuous and injective on $(0,1)$.
\end{definition}

The Jordan arc is of class $C^{n}$ ($n\in\N^*$) if its parametrization $\f$ is $n$ times continuously differentiable, satisfying $\f'(s)\neq 0$ for all $s$.

Let $\G: \G(s),0\leq s\leq 1$ be a Jordan arc. Then the subset $\R^2\setminus\G$ is connected and we will denote it by $\Pi$. The purpose of the following proposition is to give some properties of a biholomorphism $T: \Pi \to \inte\ D^c$. After applying a homothetic transformation, a rotation and a translation, we can suppose that the endpoints of the curve are $-1=\G(0)$ and $1=\G(1)$.

\begin{proposition} \label{ana_comp}
If $\G$ is a $C^2$ Jordan arc, such that the intersection with the segment $[-1,1]$ is a finite union of segments and points, then there exists a biholomorphism $T:\Pi\to \inte\ D^c$ which verifies the following properties:
\begin{itemize}
\item $T^{-1}$ and $DT^{-1}$ extend continuously up to the boundary, and $T^{-1}$ maps $S$ to $\G$,
\item $DT^{-1}$ is bounded,
\item $T$ and $DT$ extend continuously to $\G$ with different values on each side of $\G$, except at the endpoints of the curve where $DT$ behaves like the inverse of the square root of the distance.
\item $DT$ is bounded in the exterior of any disk $B(0,R)$, with $\G \subset B(0,R)$,
\item $DT$ is $L^p(\Pi\cap B(0,R))$ for all $p<4$ and $R>0$.
\end{itemize}
\end{proposition}

The proof of this proposition can be found in \cite{lac_euler}. Reading it, one understands why we need that the curve is supposed to be more that $C^1$ (in fact, $C^{1,\a}$ can be sufficient). Indeed, we want some continuity properties of the first derivate of $T$, which is possible by the Kellogg-Warschawski Theorem only if $\G$ is enough regular. We also recall from \cite{lac_euler} the following remark:
\begin{remark}\label{T_inf} If we have a biholomorphism $H$ between the exterior of a bounded set $A$ and $D^c$, such that $H(\infty)=\infty$ then there exists a nonzero real number $\b$ and a bounded holomorphic function $h:\Pi\to \C$ such that
$$H(z)=\b z+ h(z),$$
with 
$$h'(z)=O\left(\frac{1}{|z|^2}\right), \text{ as }|z|\to\infty.$$

This property can be applied for the $T$ above, observing that $T$ sends the exterior of a bounded set $B$ to $\inte B(0,2)^c$, hence $T/2= \b z+ h(z)$.
\end{remark}

\subsection{The evanescent obstacle}\label{evanescent}\

We will formulate in this subsection a precise statement of the thin obstacle problem. Many of the key issues regarding the small obstacle limit and incompressible flow have been discussed in detail in \cite{lac_euler}, so we recall briefly some properties.

As in \cite{lac_euler}, we fix $\o_0\in C^\infty_c(\R^2\setminus \G)$. Next, we introduce a family of problems, parametrized by the size of the obstacle. We consider a family of smooth domains $\O_\e$, connected, simply connected and containing $\G$, with $\e$ small enough, such that the support of $\o_0$ does not intersect $\O_\e$. Let $T_\e$ be a biholomorphism between $\Pi_\e\equiv \R^2\setminus \overline{\O_\e}$ and $D^c$, satisfying:
\begin{assumption} \label{assump}
The biholomorphism family   $\{T_\e\}$ verifies
\begin{itemize}
\item[(i)] $\|(T_\e- T)/|T|\|_{L^\infty(\Pi_\e)}\to 0$ as $\e \to 0$,
\item[(ii)] $\det(DT_\e^{-1})$ is bounded on $D^c$ independently of $\e$,
\item[(iii)] for any $R>0$, $\|DT_\e - DT\|_{L^3(B(0,R)\cap \Pi_\e)}\to 0$ as $\e \to 0$,
\item[(iv)] for $R>0$ large enough, there exists  $C_R>0$ such that $|DT_\e(x)|\leq C_R$ on $B(0,R)^c$.
\item[(v)]  for $R>0$ large enough, there exists  $C_R>0$ such that $|D^2 T_\e(x)|\leq \frac{C_R}{|x|}$ on $B(0,R)^c$.
\end{itemize}
\end{assumption}

\begin{remark} \label{remark_assump}
We can observe that property (iii) implies that for any $R$, $DT_\e$ is bounded in $L^p(B(0,R)\cap \Pi_\e)$ independently of $\e$, for $p\leq3$. Moreover, condition (i) means that $T_\e\to T$ uniformly on $B(0,R)\cap \Pi_\e$ for any $R>0$.
\end{remark}

Assumption \ref{assump} corresponds to Assumption 3.1 in \cite{lac_euler}, adding part (v) and strengthening property (i) therein. Before going on, we give an example of obstacle family.
\begin{example}
We consider $\O_\e\equiv T^{-1}(B(0,1+\e)\setminus D)$. In this case, $T_\e=\frac{1}{1+\e}T$, which verifies the previous assumption. In fact, taking Proposition \ref{ana_comp} into account $\|DT_\e - DT\|_{L^p(B(0,R)\cap \Pi_\e)}\to 0$ for all $p<4$, and using Remark \ref{remark_assump}, $|D^2 T_\e(x)|\leq \frac{C_R}{|x|^3}$ on $B(0,R)^c$, but we will not need so stronger estimates. If $\G$ is a segment, then $\O_\e$ is the interior of an ellipse around the segment.
\end{example}

We denote by $\G_\e\equiv \pd \O_\e$. Moreover, we denote by $G^\e=G^\e(x,y)$ the Green's function of the Laplacian in $\Pi_{\e}$, by $K^\e(x,y)=\na_x^\perp G^\e(x,y)$ the kernel of the Biot-Savart law on $\Pi_\e$ and we denote the associated integral operator by $f\mapsto K^\e[f]=\int_{\Pi_\e}K^\e(x,y) f(y) dy$. Let $H^\e(x)$ be the unique harmonic vector field on $\Pi_\e$ which verifies the condition 
$\oint_{\G_\e} H^\e\cdot {\bf ds}=1$, where the contour integral is taken in the counterclockwise sense. Both $K^\e$ and $H^\e$ depend on $T^\e$, and we recall explicit formulas find in the Section 3.2 of \cite{lac_euler}:
\begin{equation} \label{K} 
K^\e  = \dfrac{1}{2\pi} DT_\e^t(x)\Bigl(\dfrac{(T_\e(x)-T_\e(y))^\perp}{|T_\e(x)-T_\e(y)|^2}-\dfrac{(T_\e(x)- T_\e(y)^*)^\perp}{|T_\e(x)- T_\e(y)^*|^2}\Bigl)
\end{equation}
and 
\begin{equation}\label{H}
H^\e=\frac{1}{2\pi}DT_\e^t(x)\Bigl(\frac{(T_\e(x))^\perp}{|T_\e(x)|^2}\Bigl),
\end{equation}
where $T_\e(y)^*=\frac{T_\e(y)}{|T_\e(y)|^2}$.

We recall from \cite{ift_lop} that given  $\o_0\in C^\infty_c (\R^2\setminus \G)$ and $\g\in \R$, for $\e>0$,there exists a unique $u_0^\e$ such that $\diver u^\e_0=0$, $\curl u^\e_0=\o_0$, $\oint_{\G_\e} u_0^\e\cdot {\bf ds}=\g$, $u^\e_0$ is tangent to $\G_\e$ and vanishes at infinity. Moreover, there exists a unique $\a$ such that 
\begin{equation}\label{u_0}
u^\e_0=K^\e[\o^\e_0]+\a H^\e.
\end{equation}
By Stokes Theorem, we have that $\a=\g+m$  with $m\equiv \int_{\R^2}\o_0dx$ (see the proof of Lemma 3.1 in \cite{ift_lop}).

Now, we require information on far-field behavior. We know from Subsection 2.2 of \cite{lac_euler} that
\begin{equation*}\begin{split}
|u_0^\e(x)|\leq \frac{|DT_\e(x)|}{2\pi}\int_{\supp \o_0}\frac{|T_\e(y)-T_\e(y)^*|}{|T_\e(x)-T_\e(y)||T_\e(x)-T_\e(y)^*|}&|\o_0(y)|dy\\
&+ \frac{|DT_\e(x)|}{2\pi |T_\e(x)|}.
\end{split}\end{equation*}
Thanks to Assumption \ref{assump} (i),(iv), and to the form of $T(x)$ at infinity (see Remark \ref{T_inf}), there exist $R>0$ and $C>0$ independent of $\e$ such that  
\begin{equation}\label{H_inf}
|K^\e[\o_0](x)|\leq C/|x|^2 \text{\ and\ } |H^\e(x)|\leq C/|x|, \forall |x|\geq R,
\end{equation}
since $\o_0\in C^\infty_c(\Pi_\e)$.

Let $u^\e=u^\e(x,t)=(u_1^\e(x_1,x_2,t),u_2^\e(x_1,x_2,t))$ be the velocity of an incompressible, viscous flow in $\Pi_\e$. We assume that $u^\e$ verifies the no-slip condition at any positive time and $u^\e\to 0$ when $|x|\to \infty$. The evolution of such a flow is governed by the Navier-Stokes equations:
\begin{equation} \label{NS}
\left\lbrace\begin{aligned}
&\pd_t u^\e-\nu \D u^\e+u^\e\cdot \na u^\e=-\na p^\e &\text{ in }{\Pi_\e}\times(0,\infty) \\
&\diver u^\e =0 &\text{ in } {\Pi_\e}\times[0,\infty) \\
&u^\e =0 &\text{ in } \G_\e\times(0,\infty) \\
&\lim_{|x|\to\infty}|u^\e|=0 & \text{ for }t\in[0,\infty)\\
&u^\e(x,0)=u_0^\e(x) &\text{ in } \Pi_\e
\end{aligned}\right.
\end{equation}
 
As $u_0^\e$ is smooth, and therefore locally bounded, the behavior at infinity given in (\ref{H_inf}) allows us to observe that $u_0^\e\in L^{2,\infty}(\Pi_\e)\cap L^p(\Pi_\e)$ with $p>2$. Global-in-time well-posedness for problem (\ref{NS}) was established by Kozono and Yamazaki \cite{koz}. The existence part of Kozono and Yamazaki's result requires that the initial velocity $u_0^\e$ satisfy a smallness condition of the form
\begin{equation}\label{small_cond}
\limsup_{R\to\infty} R|\{x\in \Pi_\e\mid |u_0^\e(x)|>R\}|^{1/2}\ll 1.
\end{equation}
Since $u_0^\e$ is bounded, the limsup is always zero, for any $\e>0$. Uniqueness holds without any additional conditions.


We conclude this subsection with the definition of a cutoff function family. 
Let $\F\in C^\infty(\R)$ be a non-decreasing function such that $0\leq\F\leq 1$, $\F(s)=1$ if $s\geq 2$ and $\F(s)=0$ if $s\leq 1$. Then, for $\l\geq 2$, we introduce 
\begin{equation}\label{phi}
\F^{\e,\l}=\F^{\e,\l}(x)\equiv\F\Bigl(\frac{|T_\e(x)|-1}{\l}\Bigl).
\end{equation}
Thanks to the uniform convergence of $T_\e$ to $T$ on bounded sets (see Assumption \ref{assump} (i)), we note that the cutoff function $\F^{\e,\l}$ vanishes in a ball of radius $C_1\l$ and it is identically equal to $1$ outside a larger ball of radius $C_2\l$, with $C_1$ and $C_2$ independent of $\e$. Furthermore, the radii of the annulus where $\F^{\e,\l}$ is not constant can be made independent of $\e$.

\subsection{Asymptotic initial data}\

The purpose of this section is to study the convergence, as $\e\to 0$, of the initial velocity fields $u_0^\e$. First, we introduce some notation. For each function $f$ defined on $\Pi_\e$, we denote by $Ef$ the extension of $f$ to $\R^2$, by setting $Ef\equiv 0$ in $\Pi_\e$. If $f$ is regular enough and vanishes on $\pd \O_\e$, one has that $\na Ef = E \na f $ in $\R^2$. If $v$ is a regular enough vector field defined on $\Pi_\e$ and tangent to $\pd \O_\e$, then $\diver Ev = E \diver v $ in $\R^2$. In particular, we have $\diver E u_0^\e=0$ in $\R^2$.

The following lemmas are consequences of the case of an ideal fluid treated in \cite{lac_euler}.  

\begin{lemma}\label{u_0_est} For $2<p\leq 3$, there exists $C_p>0$, which depends only on the shape of $\G$ and $\o_0$, such that $\|E u_0^\e\|_{L^p(\R^2)}\leq C_p$.
\end{lemma}
\begin{proof} By Theorem 4.4 of \cite{lac_euler}, we state that $\|E u_0^\e\|_{L^p(S)}\leq C \|E DT_\e\|_{L^p(S)}$ for any $S\subset \R^2$. Then we can use Remark \ref{remark_assump} to observe that for any $R>0$, we can find a constant  $C_p$ such that $\|E u_0^\e\|_{L^p(B(0,R))}\leq C_p$ for $p\leq 3$. Recalling (\ref{u_0}) and (\ref{H_inf}), the desired conclusion follows since the function $x\mapsto 1/|x|$ is $L^p$ at infinity for $p>2$.
\end{proof}

\begin{lemma}\label{u_0_conv} We have that $Eu_0^\e \to K[\o_0]+\a H$ strongly in $L^2_{\loc}(\R^2)$ as $\e\to 0$, where $K$ and $H$ are defined as $K^\e$ and $H^\e$ respectively (see (\ref{K}) and (\ref{H}))by replacing $T_\e$ by $T$.
\end{lemma}

\begin{proof} This result is a consequence of Subsection 5.1 of \cite{lac_euler}, where it is shown that in the case of an ideal flow, $\F^\e u^\e \to u\equiv K[\o]+\a H $ strongly in $L^2_{\loc}([0,T]\times \R^2)$ with $\F^\e\equiv \F^{\e,\e}$. This has been done in two steps: first we prove that  $\F^\e u^\e \to u$ strongly in $L^2_{\loc}(\R^2)$ for each $t\geq 0$, and then the dominated convergence theorem allows to get the convergence in  $L^2_{\loc}([0,T]\times \R^2)$. Here the first step is sufficient to complete the proof.
\end{proof}

For the rest of the paper, we define 
\begin{equation}
  \label{initvel}
u_0=K[\o_0]+\a H,  
\end{equation}
with
\begin{equation}\label{Kbis}
K   = \dfrac{1}{2\pi} DT^t(x)\Bigl(\dfrac{(T(x)-T(y))^\perp}{|T(x)-T(y)|^2}-\dfrac{(T(x)- T(y)^*)^\perp}{|T(x)- T(y)^*|^2}\Bigl)
\end{equation}
and 
\begin{equation}\label{Hbis}
H =\frac{1}{2\pi}DT^t(x)\Bigl(\frac{(T(x))^\perp}{|T(x)|^2}\Bigl).
\end{equation}
By Proposition 5.7 of \cite{lac_euler}, we know that $u_0$
\begin{itemize}
\item[i)] is continuous on $\R^2\setminus \G$,
\item[ii)] is continuous up to $\G\setminus \{-1;1\}$, with different values on each side of $\G$,
\item[iii)] blows up at the endpoints of the curve like $C/\sqrt{|x-1||x+1|}$, which belongs to $L^p_{\loc}$ for $p<4$.
\item[iv)] is tangent to the curve.
\end{itemize}
Moreover, the Subsection 5.2 of \cite{lac_euler} states also that $u_0$ is a divergence free vector field, vanishing at infinity, with $\curl u_0=\o_0+g_{\o_0}(s)\d_\G$ in $\R^2$, where $\d_\G$ is the Dirac function of the curve $\G$, and the $g_{\o_0}$ depends on $\o_0$ and the circulation $\g$. One can also characterize $g_{\o_0}$ as the jump of the tangential velocity across $\G$. Then we know that $u_0$ is bounded except at the endpoints where it is equivalent to the inverse of the square root of the distance, and so $u_0$ verifies the smallness condition (\ref{small_cond}).

\section{Velocity estimates}

We start by introducing some functional spaces which embed the divergence-free and no-slip conditions.
\begin{definition} \label{spaces} Let $\O$ be an open set in $\R^2$. We denote by $V (\O)$ the space of divergence-free vector fields, the components of which belong to $C_c^\infty(\O)$.
The closure of $V (\O)$ in $H^1(\O)$ is denoted by $\Vc(\O)$, and its dual space by $\Vc'(\O)$.
Finally, we denote by $\Hc(\O)$ the closure of $V (\O)$ in $L^2(\O)$.
To simplify the notation, we also set
$\Vc_\G\equiv \Vc (\R^2\setminus \G)$ and $\Hc_\G\equiv \Hc (\R^2\setminus \G)$.
\end{definition}

Since the initial data $u_0^\e$ does not belong to $L^2$ ($u_0^\e= O(1/|x|)$ at infinity), we will remove the harmonic part at infinity. To this end, we denote $W^\e(t,x)= u^\e(t,x)-v^\e(x)$, where $v^\e=\a H^\e \F^{\e,\l}$, with fixed $\l$, chosen to be sufficiently large so that the radii of the balls where $\F^{\e,\l}$ vanishes, for each $\e>0$, are large enough to satisfy Assumption \ref{assump} (iv),(v). This choice of $\l$ is possible because the radii of these balls are $O(\l)$. Without any loss of generality, we may assume in addition that these balls contain $\overline{\O_\e}$. Thanks to Assumption \ref{assump} and (\ref{H_inf}), we can deduce the following estimates on $v^\e$.
\begin{lemma} \label{v_est} For $\l$ fixed (large enough independently of $\e$), we have
\begin{itemize}
\item[(a)] $v^\e$ are bounded in $L^4(\R^2)$ independently of $\e$
\item[(b)] $\na v^\e$ are bounded in $L^2(\R^2)$ independently of $\e$
\item[(c)] $\D v^\e$ are bounded in $L^\infty(\R^2)$ independently of $\e$ and supported in a compact set independent of $\e$.
\end{itemize}
\end{lemma}

\begin{proof} We recall the explicit formula of $v^\e$:
$$v^\e(x) = \frac{\a}{2\pi}\F^{\e,\l}(x) DT_\e^t(x)\Bigl(\frac{(T_\e(x))^\perp}{|T_\e(x)|^2}\Bigl),$$
with $\F^{\e,\l}$ given in (\ref{phi}).

As $\F^{\e,\l}$ vanishes in a ball of radius $O(\l)$, the conditions (i) and (iv) of Assumption \ref{assump} guarantee that $v^\e$ is uniformly bounded by $C \F^{\e,\l}(x)/|T(x)|$ for sufficiently large $\l$. Since the function $T$ behaves like $ \b x$ at infinity, the first estimate of the lemma is a consequence of the fact that $1/|x|$ is $L^4$ at infinity.

Using that $|T_\e|\geq 1$, we obtain that 
$$|\na v^\e| \leq \frac{\a}{2\pi\l}\Bigl|\F'\Bigl(\frac{|T_\e(x)|-1}{\l}\Bigl)\Bigl| |DT_\e|^2 + \frac{3\a}{2\pi}\F^{\e,\l}(x) \Bigl(\frac{|D^2T_\e|}{|T_\e(x)|}+\frac{|DT_\e|^2}{|T_\e(x)|^2}\Bigl).$$
Taking into account that the radii of the annulus where $\F^{\e,\l}$ is not constant can be made independent of $\e$, Assumption \ref{assump} (iv) implies that the first term in above inequality is uniformly bounded with respect to $x$ and $\e$, and compactly supported in a compact independent of $\e$. Parts (i), (iv) and (v) of Assumption \ref{assump} allow us to state that, for sufficiently large $\l$, the second term is bounded by $C\F^{\e,\l}(x)/|x|^2$ (with a constant $C$ independent of $\e$), which belongs to $L^2(\R^2)$. This proves the second assertion of the lemma.

Finally, we remark that $\D H^\e=0$ outside the balls where $\F^{\l,\e}$ vanish, because $H^\e=\na^\perp \ln |T_\e(x)|=\na^\perp \Re (\ln T_\e(x))$, with $\ln T_\e$ an holomorphic function, so $\D\ln T_\e =0$. Then, since $|T_\e(x)|\geq 1$, for some constant $C>0$ we have
\begin{equation*}\begin{split}
|\D v^\e|\leq C\Bigl|\F'\Bigl(\frac{|T_\e(x)|-1}{\l}\Bigl)\Bigl| ( |DT_\e|^3+ |D&T_\e| |D^2T_\e|) \\
&+C\Bigl|\F''\Bigl(\frac{|T_\e(x)|-1}{\l}\Bigl)\Bigl| |DT_\e|^3
\end{split}\end{equation*}
which is bounded in $L^\infty(\R^2)$ uniformly with respect to $\e$, and compactly supported in a compact independent of $\e$.
\end{proof}

\begin{lemma} \label{W_0_est} We have that $W^\e_0\equiv W^\e(.,0) = K_\e[\o_0]+\a(1-\F^{\e,\l})H_\e$ is bounded in $L^p$ independently of $\e$, for $1<p\leq 3$.
\end{lemma}

\begin{proof} This lemma can be established as in Lemma \ref{u_0_est} using that $W^\e_0$ behaves like $1/|x|^2$ at infinity (see (\ref{H_inf})), which belongs to $L^p$ for $p>1$.
\end{proof}
In particular, $W^\e_0$ is bounded in $L^2$, which will be useful in getting {\it a priori} estimates for $W^\e\equiv u^\e-v^\e$.

\begin{lemma}\label{est} The vector fields $W^\e$ are bounded independently of $\e$ in $L^{\infty}_{\loc}([0,\infty);L^2(\Pi_\e))\cap L^2_{\loc}([0,\infty);H^1(\Pi_\e))$.
\end{lemma}

\begin{proof} We rewrite (\ref{NS})  for $W^\e$ as follows
\begin{equation*}
\left\lbrace \begin{aligned}
& \begin{split} \pd_t W^\e-\n\D W^\e-\n\D v^\e+(W^\e+v^\e)\cdot \na W^\e&+W^\e\cdot \na v^\e+v^\e\cdot \na v^\e \\ &=-\na p^\e \text{ in } \Pi_\e\times (0,\infty)\end{split}\\
& \diver W^\e=0 \text{\ \ \ \  in } \Pi_\e\times [0,\infty)\\
& W^\e(\cdot,t)=0 \text{\ \ \ \  on }  \G_\e \times (0,\infty)
\end{aligned} \right .
\end{equation*}
Indeed, $\diver W^\e=-\diver v^\e=\a H^\e\cdot \na \F^{\e,\l}=-\frac{\a}{2\pi\l}(T_\e/|T_\e|^2 DT_\e)^\perp \cdot \F'(\frac{|T_\e|-1}{\l})(T_\e/|T_\e| DT_\e)=0$.
We multiply the equation above by $W^\e$ and integrate to obtain
\begin{eqnarray*}
\mathcal{E} &\equiv &\frac{1}{2}\frac{d}{dt}\|W^\e\|^2_{L^2}+\n\|\na W^\e\|^2_{L^2} \\
&=& -\int_{\Pi_\e} [W^\e\cdot (W^\e\cdot \na v^\e)+W^\e\cdot (v^\e\cdot \na v^\e)]dx+\n \int_{\Pi_\e} W^\e \cdot \D v^\e dx \\
&=& \int_{\Pi_\e} [v^\e\cdot (W^\e\cdot \na W^\e)+v^\e\cdot (v^\e\cdot \na W^\e)]dx+ \n \int_{\Pi_\e} W^\e \cdot \D v^\e dx\\
&\leq& \|W^\e\|_{L^4}\|\na W^\e\|_{L^2}\|v^\e\|_{L^4}+\|\na W^\e\|_{L^2}\|v^\e\|_{L^4}^2+\n \| W^\e\|_{L^2}\|\D v^\e\|_{L^2}.
\end{eqnarray*}
Next, we will use the interpolation inequality: $$\|W^\e\|_{L^4}\leq C \|W^\e\|_{L^2}^{1/2}\|\na W^\e\|_{L^2}^{1/2},$$ with a constant $C>0$ independent of $\e$. This inequality in the case of $\R^2$ can be found in Chapter 1 of \cite{interp}. To obtain the corresponding inequality in $\Pi_\e$, one simply extends $W^\e$ to $\R^2$ by setting it identically zero inside $\O_\e$. As $W^\e$ vanishes on $\G_\e$, the extension has $H^1$-norm in the plane identical to the $H^1$ norm of $W^\e$ in $\Pi_\e$. Moreover, $\D v^\e$ is bounded in $L^2$ and $v^\e$ is uniformly bounded in $L^4$ independently of $\e$ thanks to Lemma \ref{v_est}. Hence,
\begin{eqnarray*}
\mathcal{E} &\leq& C\|W^\e\|_{L^2}^{1/2}\|\na W^\e\|_{L^2}^{3/2}\|v^\e\|_{L^4}+\|\na W^\e\|_{L^2}\|v^\e\|_{L^4}^2+\n \| W^\e\|_{L^2}\|\D v^\e\|_{L^2}\\
&\leq& \frac{\n}{2}\|\na W^\e\|_{L^2}^2+C_1\|W^\e\|_{L^2}^2+C_2,
\end{eqnarray*}
for some constants $C_1$ and $C_2$ independent of $\e$, so 
$$\frac{d}{dt}\|W^\e\|^2_{L^2}+\n\|\na W^\e\|^2_{L^2}\leq 2C_1\|W^\e\|_{L^2}^2+2C_2.$$ 
Gronwall's inequality now gives, for any $t>0$,
\begin{equation}
\label{est_energie_1}
e^{-2C_1t}\|W^\e\|^2_{L^2}+\n\int_0^t e^{-2C_1s}\|\na W^\e(s,\cdot)\|^2_{L^2}ds\leq \frac{C_2}{C_1}+\|W^\e(0,\cdot)\|^2_{L^2}.
\end{equation}
 
Using the fact that $W^\e(0,\cdot)$ are bounded in $L^2$ independently of $\e$ (see Lemma \ref{W_0_est}), we can rewrite (\ref{est_energie_1}) as
 \begin{equation}
\label{est_energie}
\|W^\e\|^2_{L^2(\Pi_\e)}+\n e^{2C_1t}\int_0^t e^{-2C_1s}\|\na W^\e(s,\cdot)\|^2_{L^2(\Pi_\e)}ds\leq e^{2C_1t}C,
\end{equation}
with a constant $C$. This completes the proof.
\end{proof}

We now deduce the main result of this section.

\begin{theorem} \label{u_est}
Let $u^\e$ be the solution of (\ref{NS}), then the following hold true.
\begin{itemize}
\item[1.] The family $\{Eu^\e-v^\e\}$ is bounded in $L^\infty_{\loc} ((0,\infty);L^2(\R^2))\cap L^2_{\loc} ([0,\infty);H^1(\R^2))$.
\item[2.] The family $\{\na Eu^\e\}$ is bounded in $L^2_{\loc} ([0,\infty);L^2(\R^2))$.
\item[3.] The family $\{Eu^\e\}$ is bounded in \newline $L^\infty_{\loc} ((0,\infty);L^2_{\loc}(\R^2))\cap L^4_{\loc}([0,\infty);L^4(\R^2))$.
\end{itemize}
\end{theorem}
\begin{proof} The proof is based on Lemmas \ref{v_est} and \ref{est}. Indeed, part 1. follows from Lemma \ref{est}, while part 2. is a consequence of the same lemma and of Lemma \ref{v_est} (b). To prove part 3., we use again the interpolation inequality $\|W^\e\|_{L^4(L^4)} \leq C\|W^\e\|_{L^\infty(L^2)}^{1/2}\|\na W^\e\|_{L^2(L^2)}^{1/2}$ which ensures that $W^\e$ is uniformly bounded in $L^4_{\loc}([0,\infty);L^4(\R^2))$. It suffices now to use Lemma \ref{v_est} (a), which give the uniform boundedness in $L^4_{\loc}([0,\infty);L^4(\R^2))$ for $u^\e$ (whereas $Eu^\e_0$ is not uniformly bounded in $L^4_{\loc}(\R^2)$).
\end{proof}

For each $\e>0$, we know that $\diver E W^\e = \diver E u^\e=0$ on $\R^2$. Moreover, since the supports of  $E W^\e$ and $E u^\e$ are contained in $\Pi_\e$, we can transpose the previous theorem with the functional spaces of Definition \ref{spaces}.

\begin{corollary} \label{u_est_2}
Let $u^\e$ be the solution of (\ref{NS}), then the following hold true.
\begin{itemize}
\item[1.] The family $\{Eu^\e-v^\e\}$ is bounded in \newline $L^\infty_{\loc} ((0,\infty);\Hc_\G) \cap L^2_{\loc} ([0,\infty);\Vc_\G)$.
\item[2.] The family $\{\na Eu^\e\}$ is bounded in $L^2_{\loc} ([0,\infty);\Hc_\G)$.
\end{itemize}
\end{corollary}

We will later use the following proposition on regularization of functions in  $L^2_{\loc} ([0,\infty);\Vc_\G)$.

\begin{proposition}\label{approx} Let $T\in [0,+\infty)$ and $f\in L^2 ([0,T];\Vc_\G)$. There exists a sequence $\{f_n\}$ of divergence-free functions belonging to $C^\infty_c((0,T)\times (\R^2\setminus \G))$ such that $f_n\to f$ in $L^2([0,T],\Vc_\G)$.
\end{proposition}

\begin{proof} In order to find this family, we start by regularizing in time as done in \cite{temam}. To this end, we multiply $f$ by the characteristic function $\chi_{[1/n,T-1/n]}$ and then regularize by a function $\r_n(t)$ such that the size of the support of $\r_n$ is less or equal than $1/(2n)$. Therefore we obtain a family $\{\r_n * (\chi_{[1/n,T-1/n]} f)\}$ which belongs to  $C^\infty_c((0,T), \Vc_\G)$ and which tends to $f$ in $L^2([0,T],\Vc_\G)$. Now, we will approach functions in $C^\infty_c(\Vc_\G)$ by divergence-free functions in $C^\infty_c((0,T)\times (\R^2\setminus \G))$, which will allow us to conclude thanks to a diagonal extraction of a subsequence.

As $\Vc_\G$ is a separable Hilbert space for the scalar product $H^1(\R^2)$, $\Vc_\G$ admits an orthonormal base $\{e_n\}$. Let $\f_{n,m}\in V(\R^2\setminus \G)$ be a sequence tending to $e_n$ in $H^1(\R^2)$ as $m\to \infty$. Clearly, the family $\{ \f_{n,m}\}$ is countable, and the vector space generated by this family is dense in $\Vc_\G$. Therefore, by Gram-Schmidt we can conclude that there exists an  orthonormal base $\{\tilde e_n\}$ of $\Vc_\G$ with $\tilde e_n \in V(\R^2\setminus \G)$. So, if $f\in C^\infty_c((0,T);\Vc)$, we can write $f=\sum \a_n(t) \tilde e_n(x)$ with $\a_n\in C^\infty_c((0,T))$, and we can choose
$$f_N=\sum_0^N \a_n(t) \tilde e_n(x).$$
Those functions belong to $C^\infty_c((0,T)\times (\R^2\setminus \G))$. Moreover, $g_n(t)=\|f(\cdot,t)-f_n(\cdot,t)\|^2_{H^1}$ belongs to $L^1([0,T])$ (since $\| g_n \|_{L^1}\leq 4 (\|f\|_{L^2([0,T],H^1)})^2$), and for each $t\in [0,T]$, $\{g_n(t)\}$ is a non-increasing sequence, which tends to zero.  Then by the Beppo Levi theorem, $g_n$ tends to zero in $L^1([0,T])$, which means that $f_n$ converges to $f$ in $L^2([0,T],H^1(\R^2))$.
\end{proof}

\section{Passing to the limit}

In this section, we prove that $\{ Eu^\e \}$ converges to a solution of the Navier-Stokes equations on $\R^2\setminus\G$ in the sense of distributions.  It suffices to find a strong convergence for the sequence $\{ Eu^\e \}$ in $L^2_{\loc}([0,\infty)\times (\R^2\setminus\G))$.

\begin{proposition} \label{com_w} Let $T>0$ and let $O$ be a smooth open set relatively compact in $\R^2\setminus \G$. Then the sequence $\{ Eu^\e \}$ is precompact in $L^\infty((0,T);H^{-3}(O))$.
\end{proposition}

\begin{proof} We show that $\{ Eu^\e\}$ is bounded in $L^\infty((0,T);L^2(O))$ and equicontinuous as a function of $(0,T)$ into $H^{-2}(O)$, which will allow us to apply Arzela-Ascoli's Theorem. Fix $\P$ a smooth divergence-free vector field, compactly supported in $O$. As the obstacle shrinks to the curve $\G$, there exists $\e_O>0$ such that $\O_\e\cap \overline{O} = \emptyset$ for all $0<\e\leq \e_O$. For each interval $(t_1,t_2)\subset (0,T)$, using (\ref{NS}) we see that
\begin{eqnarray*}
\langle Eu^\e(t_2)-Eu^\e(t_1),\P\rangle&=&\int_{\R^2}(Eu^\e(t_2)-Eu^\e(t_1))\P dx\\
&=&\int_{\R^2}\bigl(\int_{t_1}^{t_2} \pd_t Eu^\e dt\bigl)\P dx \\
&=&-\int_{t_1}^{t_2}\int_{\R^2} Eu^\e\cdot \na u^\e\P \,dx\,dt\\
&&-\n \int_{t_1}^{t_2}\int_{\R^2} \na u^\e \na\P \,dx\,dt \\
&\equiv& I_1+I_2.
\end{eqnarray*}
We first estimate $I_1$. Using Theorem \ref{u_est}, we deduce that
\begin{eqnarray*}
|I_1| &\leq& \|Eu^\e\|_{L^\infty((0,T);L^2(O))} \|\na Eu^\e\|_{L^2 ([0,T];L^2(O))}\|\P\|_{L^\infty}\sqrt{|t_2-t_1|}\\
&\leq& C\|\P\|_{H^2}\sqrt{|t_2-t_1|},
\end{eqnarray*}
thanks to the Sobolev embedding $H^2(\R^2) \hookrightarrow  L^\infty(\R^2)$. Next, we treat $I_2$:
$$|I_2|\leq \n \|\na u^\e\|_{L^2([0,T];L^2(O))}\|\na \P\|_{L^2}\sqrt{|t_2-t_1|}\leq C\|\P\|_{H^2}\sqrt{|t_2-t_1|}.$$
The above inequalities show that $\{Eu^\e\}$ is equicontinuous as a function of time into $H^{-2}(O)$.
 
Since $\{ Eu^\e\}$ is bounded in $L^\infty((0,T);L^2(O))$ by Theorem \ref{u_est}, it follows from Arzela-Ascoli's theorem that there is a subsequence of $Eu^\e$ which converges strongly in $L^\infty((0,T);H^{-3}(O))$.
\end{proof}

We now improve the space-time compactness result, which is a direct consequence of the previous proposition.

\begin{lemma} \label{com}
There exists a sequence such that $\{ E u^\e\}$ converges strongly in $L^2_{\loc}([0,\infty)\times (\R^2\setminus\G))$.
\end{lemma}

\begin{proof}
We know from Theorem \ref{u_est} that $\{ E u^\e\}$ is bounded in \newline $L^2([0,T];H^1(O))$, and Proposition \ref{com_w} states that $\{ E u^\e\}$ is precompact in $L^\infty((0,T);H^{-3}(O))$. It follows by interpolation that there exists a subsequence such that $\{ E u^\e\}$ converges strongly in $L^2([0,T]\times O)$ . By taking diagonal subsequences in the set of the compact subset of $\R^2\setminus \G$ and in the time, we may assume that there is a subsequence which converges strongly in $L^2_{\loc}([0,\infty)\times (\R^2\setminus\G))$.
\end{proof}

We will prove that the limits of the sequence $\{ E u^\e\}$ are solutions of the Navier-Stokes equations on the exterior of a curve in a suitable weak sense. The difficulty is that $Eu^\e$ does not belong to $L^2(\R^2)$. So, as we did in Corollary \ref{u_est_2}, we should keep the harmonic part $v^\e$. Since we previously obtained a limit for $Eu^\e$, now we look for a limit for $v^\e$. We recall that $v^\e = \a H_\e \F^{\e,\l}$, with $H_\e$ and $\F^{\e,\l}$ are given in (\ref{H}) and (\ref{phi}). We also define $H$ and $\F^{0,\l}$ as $H^\e$ and $\F^{\e,\l}$ by replacing $T_\e$ by $T$.

\begin{lemma}\label{v_conv} If we denote $v\equiv \a H \F^{0,\l}$, then $v_\e \to v$ in $L^2_{\loc}(\R^2)$.
\end{lemma}
\begin{proof} For any compact $K$ of $\R^2$, using the explicit formula of $v^\e$ and $v$, we have
\begin{eqnarray*}
\|v^\e -v \|_{L^2(K)}&=&  \frac{\a}{2\pi}\Bigl\| \F^{\e,\l}  \Bigl(DT_\e^t  \frac{T_\e^\perp }{|T_\e |^2}-DT^t  \frac{T^\perp }{|T |^2}\Bigl)\\
&& \hspace{3cm} + (\F^{\e,\l}  -\F^{0,\l} )\Bigl(DT^t  \frac{T^\perp }{|T |^2}\Bigl)  \Bigl\|_{L^2(K)}\\
&\leq&  \frac{\a}{2\pi}\Bigl\| \F^{\e,\l}  \Bigl(DT_\e^t  \frac{T_\e^\perp }{|T_\e |^2}-DT^t  \frac{T^\perp }{|T |^2}\Bigl)\Bigl\|_{L^2(K)} \\
&&+ \frac{\a}{2\pi}\|\F^{\e,\l}  -\F^{0,\l} \|_{L^\infty} \|DT^t \|_{L^2(K)}.
\end{eqnarray*}
Recalling that $\F^{\e,\l}=0$ on a ball of radius $C_1 \l$, then from Assumption \ref{assump} (iii) and Remark \ref{remark_assump}, we can conclude that the first term tends to zero. For the second one, we note that the cutoff function $\F$ is Lipschitz, and by the explicit formula of $\F^{\e,\l}$ given in (\ref{phi}) we conclude that
$$|\F^{\e,\l}(x) -\F^{0,\l}(x)| \leq (\sup |\F'|) \Bigl|\frac{|T_\e(x)|-|T(x)|}{\l}\Bigl|.$$
Then, on the annulus (chosen independent of $\e$) where $\F^{\e,\l}-\F^{0,\l}$ is not zero, the previous term tends to zero thanks to Remark \ref{remark_assump}.
\end{proof}

Therefore we can formulate precisely the notion of weak solution we will use.
\begin{definition} \label{def} Let $u_0$ be such that $u_0-v\in \Hc_\G$. We say that $u$ is a weak solution of the incompressible Navier-Stokes equations  on $\R^+\times (\R^2\setminus \G)$ with initial velocity $u_0$ if and only if $u-v$ belongs to the space
$$C([0,\infty);\Hc_\G)\cap L^2_{\loc}([0,\infty);\Vc_\G)$$
and for any divergence-free test vector field $\p\in C_c^\infty((0,\infty)\times (\R^2\setminus \G))$, the vector field $u$ satisfies the following condition:
\begin{equation}\label{eq1}
\int_0^\infty\int_{\R^2\setminus\G}( u\cdot \p_t+[(u\cdot \na)\p]\cdot u+\nu u\cdot \D \p)\,dx\,dt=0.
\end{equation}
Furthermore, $\diver u=0$ in the sense of distributions, and $u(\cdot,t)\rightharpoonup u_0$ in the sense of distributions as $t\to 0^+$.
\end{definition}

\begin{remark} \label{est_t} In fact, if we prove that the vector field $u$ verifies (\ref{eq1}) for all divergence-free test vector fields $\p\in C_c^\infty((0,\infty)\times (\R^2\setminus \G))$, with $u-v$ belonging to $L^2_{\loc}([0,\infty);\mathcal{V}_\G)\cap L^\infty_{\loc}((0,\infty);\mathcal{H}_\G)$  then
\begin{equation}\label{est_t2}
\pd_t u\in L^2_{\loc}([0,\infty),\Vc_\G').
\end{equation}
Indeed, with Lemma \ref{v_est} and the interpolation inequality $\|u-v\|_{L^4(L^4)} \leq C\|u-v\|_{L^\infty(L^2)}^{1/2} \|\na(u-v)\|_{L^2(L^2)}^{1/2}$, we remark that $u$ belongs to \newline $L^4_{\loc}([0,\infty);L^4(\R^2\setminus\G))$ and $\na u$ belongs to $L^2_{\loc}([0,\infty);L^2(\R^2\setminus\G))$.For each $T>0$, using (\ref{eq1}) and Theorem \ref{u_est} for each divergence-free function $\p\in C_c^\infty((0,T)\times (\R^2\setminus \G))$, we have
\begin{eqnarray*} \langle \pd_t u, \p\rangle &\leq& (\|u\|^2_{L^4((0,T);L^4)}+ \n \|\na u\|_{L^2((0,T);L^2)}) \|\na \p\|_{L^2((0,T);L^2)} \\
&\leq& C \|\p\|_{L^2((0,T);\Vc_\G)}
\end{eqnarray*}
with a constant $C>0$. As the set of divergence-free function belonging in
$C_c^\infty((0,T)\times (\R^2\setminus \G))$ is dense on $L^2([0,T],\Vc_\G)$ (thanks to Proposition \ref{approx}), then the linear form $\p\mapsto \int\int \pd_t u\cdot  \p$ is bounded on $L^2([0,T],\Vc_\G)$, so (\ref{est_t2}) follows.
\end{remark}

\begin{theorem} \label{main}
There exists one strong limit $u$ of $\{ E u^\e\}$ in $L^2_{\loc}([0,\infty)\times (\R^2\setminus\G))$ which is a weak solution of the Navier-Stokes equations in $\R^2\setminus\G$ in the sense of Definition \ref{def}, with initial velocity given by $u_0=K[\o_0]+\a H$.
\end{theorem}

\begin{proof} By Lemmas \ref{u_0_conv} and \ref{v_conv}, we know that $Eu_0^\e-v^\e\to u_0-v$ in $L^2_{\loc}(\R^2)$. According to Theorem \ref{u_est}, $u_0-v$ belongs to $L^2(\R^2)$. Moreover, $Eu_0^\e-v^\e$ is supported in a smooth domain ($\Pi_\e$), then we can approach it by functions in $V_\G$. Then, by a diagonal extraction, we obtain that $u_0-v\in \Hc_\G$.

Let $\p\in C_c^\infty((0,\infty)\times (\R^2\setminus \G))$, such that $\diver \p=0$. If we consider $\e$ small enough such that the support of $\p$ does not intersect $\O_\e$, we can rewrite the integrals on $\Pi_\e$ as full plane integrals, using the extension operator and multiplying (\ref{NS}) by $\p$, we obtain the following relation:
$$\int_0^\infty\int_{\R^2\setminus\G}( Eu^\e\cdot \p_t+[(Eu^\e\cdot \na)\p]\cdot Eu^\e+\nu Eu^\e\cdot \D \p)\,dx\,dt=0.$$
Thanks to the convergence of $Eu^\e$ to a vector field $u$ in $L^2_{\loc}([0,\infty)\times \R^2\setminus\G)$ (see Lemma \ref{com}), we can pass to the limit $\e\to 0$ and obtain that $u$ satisfies (\ref{eq1}).

Moreover, $v^\e$ tends to $v$ (see Lemma \ref{v_conv}) so, passing to a subsequence if necessary, Corollary \ref{u_est_2} implies that $u-v$ belongs in $L^2_{\loc}([0,\infty);\mathcal{V}_\G)\cap L^\infty_{\loc}((0,\infty);\mathcal{H}_\G)$. The incompressible condition is a consequence of the strong convergence of divergence-free vector fields (Lemma \ref{com}).

Now, we prove that $u-v$ belongs to $C([0,\infty);\Hc_\G)$. We know from Corollary \ref{u_est_2} that $u-v$ belongs to $L^2([0,T];\Vc_\G)$ and from  Remark \ref{est_t} that its derivative $\pd_t (u-v)$ belongs to $L^2([0,T];\Vc_\G')$. As $\Vc_\G \hookrightarrow \Hc_\G \equiv \Hc_\G'\hookrightarrow \Vc_\G'$, then Lemma 1.2 in Chapter III of \cite{temam} (see also the theorem of interpolation of Lions-Magenes \cite{lions}) allows us to state that $u-v$ is almost everywhere equal to a function continuous from $(0,T)$ into $\Hc_\G$ and we have the following equality, which holds in the scalar distribution sense on $(0,T)$:
\begin{equation}\label{deriv}
\frac{d}{dt}|u-v|^2=2\langle \pd_t (u-v),u-v\rangle.
\end{equation}
Therefore, $u-v\in C([0,\infty);\mathcal{H}_\G)$.

Furthermore, since $Eu^\e$ converges to $u$ uniformly in time with values in $H^{-3}_{\loc}(\R^2\setminus \G)$ (by Proposition \ref{com_w}), one has that $Eu^\e_0$ converges to $u_{t=0}$ in $H^{-3}_{\loc}$. On the other hand, Lemma \ref{u_0_conv} states that $Eu^\e_0$ converges to $K[\o_0]+\a H $ in $L^2_{\loc}(\R^2)$. By uniqueness of the limit in $H^{-3}_{\loc}$, we conclude that $u_0=K[\o_0]+\a H$, which completes the proof.
 \end{proof}

\section{Uniqueness for the limit problem}

We now state the uniqueness result that completes Theorem \ref{main}.

\begin{proposition}\label{prop : unicity}
There exists at most one global solution in the sense of Definition \ref{def}, verifying that the initial velocity is $u_0=K[\o_0]+\a H$. 
\end{proposition}

\begin{proof} Let $u_1$ and $u_2$ two global solutions of the Navier-Stokes equations around the curve $\G$ with the same initial velocity $u_0=K[\o_0]+\a H$. By remark \ref{est_t} we have that $\pd_t u_i$ belong to $L^2_{\loc}([0,\infty),\Vc_\G')$, for $i=1,2$.

If we denote $\tilde u=u_1-u_2$, then by Proposition \ref{approx}, for a fixed $T>0$ there exist a divergence-free family $\{\p_n\}$ in $C^\infty_c((0,T)\times \R^2\setminus \G))$ such that  $\p_n\to \tilde u$ in $L^2([0,T];\Vc_\G)$. 

Subtracting the equations satisfied by $u_1$ and $u_2$, and multiplying by the test function $\p_n$, we see that
\begin{equation}\label{uni}\begin{split}
\int_0^T\int_{\R^2\setminus\G} \pd_t \tilde u\cdot \p_n &\,dx\,dt  -\nu \int_0^T\int_{\R^2\setminus\G} \tilde u\cdot \D \p_n\,dx\,dt\\
& = \int_0^T\int_{\R^2\setminus\G} \bigl([(\tilde u\cdot \na)\p_n]\cdot u_1+[( u_2\cdot \na)\p_n]\cdot \tilde u \bigl)\,dx\,dt.
\end{split}\end{equation}
Using the interpolation inequality $\|u^\e\|_{L^4(L^4)} \leq C\|u^\e\|_{L^\infty(L^2)}^{1/2}\|\na u^\e\|_{L^2(L^2)}^{1/2}$, the right hand side term can be bounded by 
\begin{equation*}\begin{split}
\int_0^T \|\tilde u\|_{L^4}(\| u_1 \|_{L^4}&+\|u_2\|_{L^4}) \| \na \p_n\|_{L^2} \\
&\leq C\int_0^T  \| \na \p_n\|_{L^2} \| \na \tilde u\|_{L^2}^{1/2}\|\tilde u\|_{L^2}^{1/2} (\| u_1 \|_{L^4}+\|u_2\|_{L^4}) \\
&\leq \frac{\n}{2}  \int_0^T  \| \na \p_n\|_{L^2}^2 +  \frac{\n}{2}  \int_0^T  \| \na \tilde u\|_{L^2}^2 \\
&\hspace{3cm}+C_1  \int_0^T \|\tilde u\|_{L^2}^{2} (\| u_1 \|_{L^4}^4+\|u_2\|_{L^4}^4),
\end{split}\end{equation*}
with constants $C$ and $C_1$ independent of $T$. For the left hand-side term, thanks to (\ref{deriv}) and because $\tilde u(.,0)=0$, we can write that
\begin{eqnarray*}
\int_0^T\int_{\R^2\setminus\G} \pd_t \tilde u\cdot \p_n \,dx\,dt&=&\int_0^T\int_{\R^2\setminus\G} \pd_t \tilde u\cdot \tilde u \,dx\,dt\\
&& +\int_0^T\int_{\R^2\setminus\G} \pd_t \tilde u\cdot (\p_n-\tilde u) \,dx\,dt\\
&=& \frac{1}{2}\|\tilde u(\cdot,T)\|_{L^2(\R^2)}\\
&&+\int_0^T\int_{\R^2\setminus\G}\pd_t \tilde u\cdot (\p_n-\tilde u) \,dx\,dt.
\end{eqnarray*}
The last double integral tends to zero as $n\to \infty$ because $\pd_t \tilde u$ belongs to $L^2_{\loc}([0,\infty);\Vc_\G')$ and $\p_n$ converges to $\tilde u$ in  $L^2([0,T];\Vc_\G)$.

In the same way, we have that 
\begin{eqnarray*}
-\lim_{n\to \infty} \int_0^T\int_{\R^2\setminus\G} \tilde u\cdot \D \p_n\,dx\,dt &=& \lim_{n\to \infty} \int_0^T\int_{\R^2\setminus\G} \na\tilde u\cdot \na \p_n\,dx\,dt \\
&=& \int_0^T\int_{\R^2\setminus\G} \na\tilde u\cdot \na\tilde u \,dx\,dt \\
&&+ \lim_{n\to \infty} \int_0^T\int_{\R^2\setminus\G} \na\tilde u\cdot (\na \p_n-\na \tilde u)\,dx\,dt \\
&=& \| \na \tilde u\|_{L^2([0,T],L^2(\R^2))}^2,
\end{eqnarray*}
because  $\na\tilde u$ belongs to $L^2([0,T];\Hc_\G)$ and $\na\p_n$ converges to $\na\tilde u$ in  $L^2([0,T];\Hc_\G)$. This convergence implies also that \newline $\lim_{n\to \infty} \| \na \p_n\|_{L^2([0,T]\times \R^2)}^2=\| \na \tilde u\|_{L^2([0,T]\times \R^2)}^2$.
Therefore, passing to the limit $n\to \infty$ in (\ref{uni}) yields
$$\|\tilde u(\cdot,T)\|_{L^2}^2  \leq 2 C_1 \int_0^T \|\tilde u\|_{L^2}^{2} (\| u_1 \|_{L^4}^4+\|u_2\|_{L^4}^4).$$
This last equality holds for all $T>0$, with the constant $C_1$ independent of $T$. Noting that the functions $t\mapsto \|\tilde u(\cdot,t)\|_{L^2}^2$,  $t\mapsto (\| u_1(\cdot,t) \|_{L^4}^4+\|u_2(\cdot,t)\|_{L^4}^4)$, and  $t\mapsto \|\tilde u(\cdot,t)\|_{L^2}^2 (\| u_1(\cdot,t) \|_{L^4}^4+\|u_2(\cdot,t)\|_{L^4}^4)$ are $L^1_{\loc}$, we can apply Gronwall lemma to get that 
$$\|\tilde u(\cdot,T)\|_{L^2}^2 \leq 0,$$
which concludes the proof of uniqueness.
\end{proof}

Once the uniqueness of the limit velocity is established, and given that from Theorem \ref{main} we know that from every sequence of solutions $u^\e$ we can extract a subsequence converging in $L^2_{\loc}([0,\infty)\times (\R^2\setminus\G))$, we deduce with a standard argument that strong convergence in $L^2_{\loc}([0,\infty)\times (\R^2\setminus\G))$ holds without need to extract a subsequence. Theorem \ref{intro} is therefore completely proved.

\subsection*{Acknowledgments}

I would like to thank S. Monniaux for several interesting and helpful discussions.

\section*{List of notations}

\subsection*{Domains:}\

$D\equiv B(0,1)$ the unit disk.

$S\equiv \pd D$.

$\G$ is a Jordan arc (see Proposition \ref{ana_comp}).

$\Pi \equiv \R^2\setminus \G$.

$\O_\e$ is a bounded, open, connected, simply connected subset of the plane, such as $\O_\e\to \G$ as $\e\to 0$.

$\G_\e\equiv \pd\O_\e$  is a $C^\infty$ Jordan curve and $\Pi_\e\equiv \R^2\setminus \overline{\O_\e}$.

\subsection*{Functions:}\

$\o_0$ is the initial vorticity ($C^\infty_c(\Pi)$).

$\g$ is the circulation of $u_0^\e$ on $\G_\e$ (see Introduction).

$u^\e$ is the solution of the Navier-Stokes equations on $\Pi_\e$.

$T$ is a biholomorphism between $\Pi$ and $\inte\ D^c$ (see Proposition \ref{ana_comp}).

$T_\e$ is a biholomorphism between $\Pi_\e$ and $\inte\ D^c$ (see Assumption \ref{assump}).

$K^\e$ and $H^\e$ are given in \eqref{K} and \eqref{H}

$K^\e[\o_0](x)\equiv \int_{\Pi_\e} K^\e(x,y) \o_0(y)dy$.

$\F^{\e,\l}$ is a cutoff function (see \eqref{phi}).

$V (\O)$, $\Vc(\O)$, $\Vc'(\O)$, $\Hc(\O)$, $\Vc_\G$ and $\Hc_\G$ are some vector spaces defined in Definition \ref{spaces}.

\end{document}